\documentclass[draft,Gr,amssymb]{amsart}
\usepackage{amsmath,Gr,oldlfont,amssymb}
\usepackage{ae} 
\normalsize 
\numberwithin{equation}{section}

\newtheorem{thm}                    {Theorem}

\newtheorem{prop}[equation]     {Proposition}
\newtheorem{cor}[equation]          {Corollary}

\theoremstyle{definition}

\theoremstyle{remark}
\newtheorem{remark}[equation]       {Remark}

\begin{document}
\title[Well-posedness of Einstein's Equation]{Well-posedness of Einstein's Equation with Redshift Data}
\thanks{The author thanks the Mathematics Departments of UW-Madison and UW-Oshkosh and
 the UW-Madison Physics Department for the use of their resources during the course of the present research.
The author would particularly like to thank
Prof. Daniel Chung of the UW-Madison Physics Department for his conversations and
suggestions leading to
the present article.}

\begin{abstract}{We study the solvability of a system of ordinary differential equations
derived from null geodesics of the LTB metric with data given in terms of a so-called redshift parameter. Data is introduced along these geodesics by the luminosity distance function. We check our results with luminosity distance depending on the cosmological constant and with the well-known FRW model.}\end{abstract}
\author{Christopher J. Winfield}
\email{winfield@madscitech.org}
\subjclass{83F05, 34A34}
\keywords{redshift parameter, ordinary differential equations, LTB cosmology, luminosity distance}
\maketitle
\section*{Introduction}

Resulting from the Lema\^itre-Tolman-Bondi metric
\begin{equation}ds^2 = -dt^2+\frac{R^{\prime}(t,r)^2dr^2}{1+2E(r)}+R(t,r)^2d\gO^2
\label{maineq}\end{equation}
are so-called symmetric dust solutions to the Einstein equation given by
 \begin{align}
\left(\frac{\dot{R}}{R}\right)^2&=\frac{2E}{R^2}+\frac{2M}{R^3}
\label{original1}
\\
\rho(t,r)&=\frac{M^{\prime}(r)}{R(t,r)^2R^{\prime}(t,r)}\label{original2}
\end{align}
(c.f. \cite{skchh}) for some suitable $\rho$ (energy density) where
superscript $\prime$ and $\cdot$ denote partial derivatives with
respect to $r$ and $t,$ respectively. Setting $\gs$ $\df$
$\sgn\dot{R}$, $\gd$ $\df$ $\sgn R^{\prime},$ $A\df B\sqrt{1+2E}$ and
$B\df{\gs\sqrt{2E+2M/R}},$ we study resulting system
\cite{cr,c}
\begin{align}
\frac{dr}{dz}&=\frac{\sqrt{1+2E}}{(1+z)\partial^2_{r,t}R(t,r)}
&=\frac{A}{E^{\prime}+M^{\prime}/R-MR^{\prime}/R^2} \label{dr}
\\
\frac{dt}{dz}&=\frac{-|R^{\prime}|}{(1+z)\partial^2_{r,t}R(t,r)}
&=\frac{-BR^{\prime}\gd}{E^{\prime}+M^{\prime}/R-MR^{\prime}/R^2},
\label{dt}
\end{align}
taken along null geodesics of (\ref{maineq}). Here data is given
for the function $R,$ prescribing values $R(t(z),r(z))$ along curves
given by (\ref{dr}) and (\ref{dt}). 
As a result,
corresponding solutions of this system provide
maps 
\begin{equation}(E(r),D_L(z),R_0(r))\,\,\rightarrow\,\,(r(z),t(z),M(r(z)))\label{map}\end{equation} as
introduced in \cite{cr} which we study in some detail in
this article. 

As an application of our analysis, we will consider data
given in the form
\begin{equation}
R(t(z),r(z))=\frac{D_L(z)}{(1+z)^2}=\frac{\int_1^{1+z}\cI(y)dy}{1+z}\label{R}
\end{equation} for $D_L(z)=(1+z)\int_1^{1+z}\cI(y)dy$ with $\cI(y)=1/\sqrt{\gO_{\gL}+(1-\gO_{\gL})y^3}$
for a real parameter $0\leq$ $\gO_{\gL}$ $\leq 1.$ Here, $D_L$ is
generally referred to as "luminosity distance" and, in particular
models, $\gO_{\gL}$ $=$ $\frac{\gL}{3H_0^2}$ is directly proportional to the
so-called "cosmological constant" $\gL$ \cite{c, cpt}. 
For further details of the physical and
mathematical derivation of the present problem, the author recommends
the aforementioned articles along with \cite{e, i, ks, pm, we, wmt} - to
name but a few.

This work is physically motivated by competing cosmological theories in
explaining certain observations of matter distribution and cosmic inflation. Such theories include
those of "dark energy" \cite{f}, certain metric perturbations from the FRW model
\cite{kmr,r,m}, radial inhomogeneities of the unperturbed LTB model (via $E(r),$ $R_0(r)$ and $M(r)$), and the cosmological constant (here via $D_L$) - with our work
involving the later two.
Here, we study the map (\ref{map}) mostly on purely mathematical grounds, presenting a
framework of analysis and, in a special case, estimates on resulting functions $M$ in terms of $z.$
Furthermore, we test our results for certain
functions $D_L,$ $E,$ and $R_0,$ arising from various FRW-type models,
and study singularities of $M$ as indications of (in-) compatibility of these
models.

\section{Singularities}
From the Chain Rule (c.f. equation (14) \cite{cr}) we observe that
$R^{\prime}$ takes the form $$R^{\prime}={\cal
F}(R,R_0,R_0^{\prime},E,E^{\prime},M,t) +M^{\prime}{\cal
G}(R,R_0,E,E^{\prime},M).$$  With $\dot{R}$ $=\gs\sqrt{2E+2M/R}$ and
$R_0(r)\df R(r,t_0)$ for a fixed $t_0>0,$ we restrict $R,R_0,t> 0,$
$M\geq -ER,$ $E> 0$ and set
\begin{equation}J(R,R_o,M,E,t)\df \sqrt{2}(t-t_0)-
\gs\int_{R_0}^{R}\sqrt{\frac{\gt}{\gt E + M}}d\gt=0\label{manifold}
\end{equation}
solutions of which define smooth manifolds $\fO^{\pm}$ depending on
constant $\gs=$ $\pm 1$, respectively.

We introduce notation: For a given function $f$ $=$ $f(t,r),$ depending implicitly or
explicitly on $(t,r),$ we will denote $f[z]$ $\df$ $f(t(z),r(z))$
and, with slight abuse of notation, set $\derz{f}$ $\df$
$\derz{f[z]}.$ We now set
\begin{equation}R^{\prime}=\cF+\cG\frac{dM/dz}{dr/dz}\label{Rprime}\end{equation} where, from
the chain rule, with subscript denoting the associated partial derivative,
\begin{align}-(\partial_RJ)\cF&=E^{\prime}\partial_EJ+R_0^{\prime}\partial_{R_0}J
\\
-(\partial_R J)\cG&=\partial_M J.
\end{align}
with $\xi$ $\df$ $M/R,$ $\xi^{\sharp},$ $\df$ $M/R_0,$ $\fh$
$\df$ $R_0/R,$ and
\begin{align}
J_R&=-\gs/\sqrt{E+\xi}&
J_{R_0}&=\gs/\sqrt{E+\xi^{\sharp}}\label{Jdiff}\\
J_M&=-\frac{\gs}{2}\int_1^{\fh}\frac{\gn^{1/2}d\gn}{(E\gn+\xi)^{3/2}}&
J_E&=-\frac{\gs R}{2}\int_1^{\fh}\frac{\gn^{3/2}d\gn}{(E\gn+\xi)^{3/2}}
\notag
\end{align}

Substituting (\ref{Rprime}) into equations (\ref{dr}) and
(\ref{dt}), we obtain
\begin{align}&\left(E^{\prime}-\frac{M}{R^2}\cF\right)\cdot\derz{r}+\frac{\derz{M}}{R}-\frac{M\derz{M}}{R^2}\cG=A
\label{dr2}\\
&\left(E^{\prime}-\frac{M}{R^2}\cF\right)\cdot\derz{r}\derz{t}+\left(\frac{\derz{M}}{R}-\frac{M\derz{M}}{R^2}\cG\right)\derz{t}\label{dt2}\\
&=-\gd\cdot
B\cdot\left(\cF+\frac{dM/dz}{dr/dz}\cG\right)\derz{r}\notag\end{align}
We then substitute
$$\left(E^{\prime}-\frac{M}{R^2}\cF\right)\derz{r}=A-\left(\frac{\derz{M}}{R}-\frac{M\derz{M}}{R^2}\cG\right)$$
so that equation (\ref{dt2}) becomes
\begin{equation}
A\derz{t}=-\gd\cdot B\cdot\left(\cF\derz{r}+\derz{M}\cG\right).
\label{dt3}\end{equation}
Equation (\ref{dt3}) can be verified by equations (\ref{original1}) and (\ref{original2}).

Now, equations
(\ref{dr2}), (\ref{dt3}), and (\ref{Rprime}) along with the Chain Rule result in the following system:
\begin{align}
&\left(E^{\prime}-\frac{M}{R^2}\cF\right)\derz{r}+\left(\frac{1}{R}-\frac{M\cG}{R^2}\right)\derz{M}&=A
\notag\\
&\gd B\cF\derz{r}+A\derz{t}+\gd B\cG\derz{M}&=0
\notag\\
&\cF\derz{r}+
\gs\sqrt{(2E+2M/R)}\derz{t}+\cG\derz{M}&=\derz{R}\notag
\end{align} which
we may write in matrix form as
\begin{equation}\cU\derz{\vec {X}}=\vec{Y}\label{sys}
\end{equation} for $$
\cU\df\left(\begin{array}{ccc}
  E^{\prime}-\frac{M}{R^2}\cF & 0 & \frac{1}{R}-\frac{M\cG}{R^2} \\
  \gd B\cF & A & \gd\cG B \\
  \cF & \gs\sqrt{2E+2M/R} & \cG
\end{array}\right)$$
$$\vec{X}=\left(\begin{array}{c}
  r \\
  t \\
  M
\end{array}\right),
\vec{Y}=\left(\begin{array}{c}
  A \\
  0 \\
  \derz{R}
\end{array}\right).$$

We check the invertibility of $\cU$ as we compute
\begin{align}\text{ det}\cU =&(E^{\prime}-\frac{M}{R^2}\cF)(A\cG
-\gs\gd\cG
B\sqrt{2E+2M/R})\notag\\
+&(\frac{1}{R}-\frac{M\cG}{R^2})(\gd\gs B\cF\sqrt{2E+2M/R}-\cF A)\notag\\
=&B(E^{\prime}\cG-\frac{\cF}{R})(\sqrt{1+2E}-\gs\gd \sqrt{2E+2M/R})\notag
\end{align}
From these computations we conclude
\begin{prop} \label{first} Suppose that $E,R>0$ with $\partial_tR,$ $\partial_rR,$
$\partial^2_{t,r}R\neq 0.$ Then, $\cU^{-1}$ is a smooth function of
$R,$ $R_0,$ $R_0^{\prime},$ $E,$ $E^{\prime},$ and $M$ except for
the following cases: Either
\begin{itemize}
\item[1.)]  both $\gd$ $=$ $\gs$ and $R=2M$; or,
\item[2.)] $E^{\prime}R\cG=\cF.$
\end{itemize}
\end{prop}
We may extend the domain of $\cU$ to include $-1/2<$ $E<0$, say, but
for simplicity we impose the above hypothesis throughout the rest of
this section. We continue with
\begin{prop}\label{second}
Suppose that for some $z^*>0$, $\gd[z^*]=\gs[z^*]$ and that
$\derz{R[z]}|_{z=z^*}$ $=0.$ Then, the matrix $\cU[z]$ is singular at
$z=z^*.$
\end{prop}
\begin{proof}
We have from the Chain Rule and equations (\ref{dr}) and (\ref{dt}) that
\begin{align}
\derz{R}&=R^{\prime}\frac{dr}{dz}+\dot{R}\frac{dt}{dz}\label{chain}\\
&=R^{\prime}\frac{\sqrt{1+2E}}{(1+z)\partial^2_{r,t}R}
-\dot{R}\frac{|R^{\prime}|}{(1+z)\partial^2_{r,t}R}
\notag\\
&=R^{\prime}\frac{\sqrt{1+2E}-\gd\gs\sqrt{2E+2M/R}}{(1+z)\partial^2_{r,t}R}
\notag\end{align}By our hypotheses on the partial derivatives of $R$ we may conclude
\begin{equation}\left(\gs\gd\sqrt{2E+2M/R}\right)[z^*]
=\left(\sqrt{1+2E}\right)[z^*]
\notag\end{equation}
With $\gs=\gd$ at $z=z^*$, we have that $\gs\gd=\gd^2=1$ and that
$\sqrt{2E+2M/R}$ $=$ $\sqrt{1+2E}$ so that $2M[z^*]$ $=R[z^*].$ Then from Proposition \ref{first}
we see that $\det U[z^*]=0.$
\end{proof}
We note that the type of singularity of item 1) of Proposition
\ref{first} appears analogous to that of the well-known
"Schwarzschild" singularity: It is not yet clear here if this is
merely an artifact of the specific model or if such singularities
are removable by passing to alternate coordinate systems or metrics (c.f. \textsection 31 \cite{wmt}, \textsection 6.4
\cite{wa}), taking us beyond the scope of the present article.

We may interpret item 2) of Proposition \ref{first} in terms of the
tangent bundles $T\fO^{\pm}$ (resp.) of manifolds obtained from
(\ref{manifold}). We may consider the transformation
$\phi^{\pm}:$ ${\Bbb R} \times (0,+\infty)$  $\rightarrow$ $\fO^{\pm}$
given by
$\phi^{\pm}(t,r)$ $\df$ $(R(t,r),R_0(r),E(r),M(r),t)$
and $d\phi$ as a push forward, to interpret corresponding solutions
to
$$E^{\prime}R\partial_M
J=E^{\prime}\partial_EJ+R_0^{\prime}\partial_{R_0}J
$$ as subsets $\cM^{\pm}$ of $T\fO^{\pm}$ (resp.) in coordinate form.
Let $\cT$ denote the set $(\gO^{+}\setminus\pi \cM^{+})$ $\bigcup$
$(\gO^{-}\setminus\pi \cM^{-})$ where $\pi$ denotes the natural
projection $\pi: T\cM$ $\rightarrow$ $\cM$ of a manifold $\cM$.

By calculating $\cU^{-1}\vec{Y}$ from (\ref{sys}) with $R_z$ $\df$
$\derz{R[z]},$ we arrive at the following system of ordinary
differential equations:
\begin{align}\label{ode}\derz{r}&
=\frac{\cG A R}{{
\cG}{E^{\prime}}R-{\cF}}-\frac{{R_z}\cdot(M{\cG}-R)
\sqrt{1+2E}} {R\cdot( \gd\gs\sqrt{2E+2M/R}-\sqrt{1+2E})( {
\cG}{E^{\prime}}R-{\cF})}
\\
\derz{t}&
= \frac{\gd\cdot{R_z}}{
\gd\gs\sqrt {2E+2M/R}-\sqrt{1+2E}}
\notag\\
\derz{M}&
=\frac{-{\cF}{R}A}{\cG E^{\prime}R-{\cF}} - \frac{R_z\cdot(
{R}^{2}{E^{\prime}}-{\cF}M)\sqrt{1+2E}}{R\cdot(\gd\gs\sqrt {2E+2M/R}-\sqrt{1+2E})
(\cG E^{\prime}R-{\cF})}
 \notag
\end{align}

We are ready to state
\begin{prop}\label{third}
The matrix $\cU$ is non-singular for $R\neq 2M$ provided
$(R,R_0,E,M, t)$ $\in$ $\cT.$ Indeed, if for some $z_0$ $>$ $0,$
these conditions hold for $(R,R_0,E,M,t)[z]|_{z=z_0},$ then the
system of equations (\ref{ode}) has a unique $C^{\infty}$ solution $\vec{X}[z]$
in some open interval containing $z_0.$
\end{prop}
\begin{proof} It is clear that the elements of $U$
are continuously differentiable where $\det U$ is non-zero. The
result follows by applying standard theory of ordinary differential
equations \cite{cl}.
\end{proof}
To further investigate the  solvability of the system (\ref{ode}), we compute
\begin{align}
\cG E^{\prime} R-{\cF}&=\frac{E^{\prime}R\cdot(J_E/R-J_M)+R^{\prime}_0J_{R_0}}{J_R}\label{gerf}\\
&=
\sqrt{E+\xi}\left(\frac{RE^{\prime}}{2}\int_1^{\fh}\frac{s^{1/2}(s-1)ds}{(Es+\xi)^{3/2}}
-\frac{R_0^{\prime}}{\sqrt{E+\xi^{\sharp}}}\right).\notag
\end{align}
Since $\frac{\gn}{(E\gn+\xi)^3}\leq\frac{4}{27E\xi^2}$
we find
$$ \int_1^{\fh}\frac{\gn^{1/2}(\gn-1)d\gn}{(E\gn+\xi)^{3/2}}\leq \frac{2}{\xi\sqrt{27E}}\int_1^{\fh}(\gn-1)\,d\gn
=\frac{(\fh-1)^2}{\xi\sqrt{27E}}$$
Lacking any other simplifying assumptions, we thus obtain strong criteria
for local solvability:
\begin{prop}
System (\ref{ode}) is locally solvable at any point of $\fO^{\pm}$ where
$\gd\neq \gs$ or where $2M\neq R$ if either of the following holds:
\begin{itemize}
\item [1)] $\sgn E^{\prime}\neq \sgn R_0^{\prime}$
\item [2)]
    $|\frac{E^{\prime}\cdot(R_0-R)^2}{2M\sqrt{27E}}|<|\frac{R_0^{\prime}\sqrt{R_0}}{\sqrt{ER_0+M}}|$
\end{itemize}
Indeed, given $r_0,t_0,M_0>0$ and smooth $E,$ $R,$
$R_0^{\prime}$ $>0,$ the system (\ref{ode}) has on an open interval
$I$ $\ni z_0$ a unique solution satisfying
$$\vec{X}(z_0)=\left(\begin{array}{c}
  r_0 \\
  t_0 \\
  M_0
\end{array}\right).$$
\end{prop}

\section{Decoupled equations: A Case of Constant $E$}
We consider the case of constant $E> 0$ in which we can rescale $M$ and $R$ to
assume the case $E=1,$ retaining
\begin{equation}\label{newltb}
\left(\frac{\dot{R}}{R}\right)^2=\frac{2}{R^2}+\frac{2M}{R^3}
\end{equation}
Here, equations (\ref{ode}) reduce to
\begin{align}
\derz{r}
&=\frac{-\cG A R}{{\cF}}+\frac{{R_z}\cdot(1-\frac{M}{R}{\cG})
\sqrt{3}} {(\sqrt{3}- \gs\gd\sqrt{2+2M/R}){\cF}}
\label{newode}\\
\derz{t}&= \frac{-\gd \cdot{R_z}}{
\sqrt{3}-\gs\gd\sqrt {2+2M/R}}
\notag\\
\derz{M}
&=RA + \frac{R_z\frac{M}{R}\sqrt{3}}{\sqrt{3}-\gs\gd\sqrt{2+2M/R}}
 \notag
\end{align}
with $A=\frac{\gs\sqrt{6}\sqrt{1+M/R}}{1+z}.$

For the remainder of the section we assume that $E,\gs,\gd\equiv 1$
and denote by $\cT_1$ the corresponding subset of $\cT.$ Then,
$R(t,r)< R_0(r)$ $\forall t$ $<t_0.$ And, for $\fh$ and $\xi$ as
above, we obtain
\begin{align}
\partial_MJ&=-\frac{1}{2}\int^{\fh}_1\sqrt{\frac{\gn}{\gn+\xi}}\frac{1}{\gn+\xi}\,d\gn\notag\\
\partial_{R_0}J&=\sqrt{\frac{1}{1+\xi^{\sharp}}};\,\,
\partial_{R}J=-\sqrt{\frac{1}{1+\xi}}\notag
\end{align}
with $\xi\geq \xi^{\sharp}$ and $\fh\geq 1$, so that the following hold:
\begin{align}
0\leq&\cJ_1(r,z,\xi)\df-\frac{\cG}{\cF}=\frac{\sqrt{1+\xi^{\sharp}}}{2R_0^{\prime}}
\int^{\fh}_1\sqrt{\frac{\gn}{\gn+\xi}}\frac{1}{\gn+\xi}\,d\gn
\label{moreests}\\
\leq &\frac{\sqrt{\fh(1+\xi^{\sharp})}}{R_0^{\prime}}\left(\frac{1}{\sqrt{1+\xi}}-\frac{1}{\sqrt{\fh+\xi}}\right)
\leq\frac{\sqrt{\fh}}{2R^{\prime}_0(1+\xi)}
;\notag\\
0<&\frac{1}{\cF}=\frac{1}{R_0^{\prime}}\sqrt{\frac{1+\xi^{\sharp}}{1+\xi}}\df\cJ_2(r,z,\xi)
\leq 1/R_0^{\prime};\notag\\
0<&\frac{1-\xi\cG}{\cF}=\cJ_2+\xi\cJ_1\leq
\frac{1+\sqrt{\fh}/2}{R_0^{\prime}}\notag .\end{align} Our change
of variables leads to
$$\derz{\xi}=\derz{M}/R-{R_z}\xi/{R}$$
with $A=\frac{\sqrt{6}\sqrt{1+\xi}}{1+z}$ whereby the system
(\ref{ode}) now reduces further to\begin{align}
\derz{r}&=\frac{R\cJ_1\sqrt{6}\sqrt{1+\xi}}{1+z}+\frac{\sqrt{3}R_z\cdot(\cJ_2+\xi\cJ_1)}
{\sqrt{3}-\sqrt{2+2\xi}}\notag\\
\derz{t}&=\frac{-R_z}{\sqrt{3}-\sqrt{2+2\xi}}\label{newerode}\\
\derz{\xi}&=\frac{\sqrt{6}\sqrt{1+\xi}}{1+z}
+\xi\frac{R_z}{R}\left(\frac{\sqrt{3}\sqrt{1+\xi}}{\sqrt{3}-\sqrt{2+2\xi}}\right).\notag\end{align}
Here, we note that the equation for $\derz{\xi}$ decouples from the
others, allowing for $\xi$ to be solved for explicitly in $z.$ Then,
with the solution to $\xi(z)$ in hand, both $\cI_1$ and $\cI_2$ depend only on
$z$ and $r$ whereby the remaining equations are then decoupled.

We give estimates for the system
(\ref{newerode}) assuming uniform bounds on
$R,$ $M/R,$ $R_z,$ $R_0$, and $R^{\prime}_0$. We suppose the
following bounds hold for $0<$ $z_0\leq z$ $\leq z_1$ and $0< r,M,t$
on some compact sets (to be determined): $\xi$ $\leq \xi^*$ with $|2\xi-1|$ $\geq\epsilon$ $>0$;
 $\rho_{min}$ $\leq$ $R$ $\leq$ $\rho_{max};$ $|R_z|$ $\leq$
$\gl$ ; $1<$ $\fh\leq\fh^*;$ and, $|R^{\prime}_0|$ $\geq$
$\fr>0.$ Here, applying (\ref{moreests})
\begin{align}
|\derz{t}|&\leq \frac{\gl}{\sqrt{3}-\sqrt{2+2\xi^*}}
=\frac{\gl\cdot(\sqrt{3}+\sqrt{2+2\xi^*})}{\epsilon}\df
\fM_1\label{conds}
\\
|\derz{r}|&\leq \sqrt{3}\frac{\rho_{max}\sqrt{\fh^*/2}
+(1+\sqrt{\fh^*}/2)\fM_1}{\fr}
\df\fM_2
\notag\\
|\derz{\xi}|&\leq\sqrt{3(1+\xi^*)}(\sqrt{2}
+{\fM_1\xi^*}/{\rho_{min}})
\df \fM_{3}\notag
\end{align}

Let $\fM\df$ $\max_j\{\fM_j\}_{j=1}^3$ and suppose $r_0,$ $t_0,$ and $M_0/R[z_0]$ $\df\xi_0$ $\neq 1/2$
satisfy the restrictions on $(r,t,\xi)$ for some $0<$ $\xi_0$ $<\xi^*$ as above with
\begin{equation}\vec{X}_0=\vec{X}(z_0)=\left(
\begin{matrix} r_0\\t_0\\M_0\end{matrix}\right)
\label{newinit}\end{equation} For an interval $I$ of the form $0\leq$ $z_0$ $\leq$
$z$ $\leq$ $z_1,$ the following now results from standard theory of differential equations \cite{cl}:
\begin{prop}
\label{est1} For $z_0\geq 0,$ the system (\ref{newode})
is solvable on an interval of the form $I=$ $\{z|z_0\leq z \leq z_1\}$
provided that the conditions (\ref{newinit}) and (\ref{conds}) hold for $\vec{X}$
in subset of $\cT_1$
given by $|(\vec{X}-\vec{X}_0)_j|$ $\leq b:$ $j=$ $1,2,3$ for some constant $b$ $<1/\fM.$
Here, a unique solution may be computed by the method of successive approximations.
\end{prop}
\begin{proof} We may apply Theorem 3.1, Chapt. 1 \cite{cl}:
The conditions assure Lipshitz continuity of the right-hand sides of
(\ref{newerode}) and that both $z-z_0$ and
$|\vec{X}-\vec{X}_0|/\fM$ are bounded above by $|z_1-z_0|,$ so that
the result follows.
\end{proof}

Recalling that we set $E\equiv 1,$ we will suppose for the rest of
the section that $R[z]$, $R^{\prime}_0[z]$ are smooth and positive
for $z>0.$ For some of our analysis below we will suppose also that
\begin{equation}\label{upprbnd}R[z]>Cz|R_z|.
\end{equation}
holds on some real interval.
We now present our estimates on $M[z]$ depending on $R[z]$ and initial conditions
given by $\xi_0\df$ $\xi(r(z_0),z_0).$
\begin{thm}\label{thm1}
Suppose that (\ref{upprbnd}) holds on some
interval $I$ $=[z_0,z_1)$ $\subset\Bbb{R^+}$. Then the following statements hold
for some constants $0<$ $c_1$ $<1/2$ $<c_2,$ each depending on the choice of $C$:
 \begin{itemize}
\item [1)] If $0<\xi_0$ $<1/2$ and $R_z$ $<0$, then
$M[z]\leq c_1R[z]$
holds on $I.$
\item[2)] If $\xi_0$ $>1/2$ and
$R_z$ $>0$, then
$M[z]\leq c_2R[z]$ on $I.$
\end{itemize}\end{thm}
\begin{proof}
Let us choose $C<1/2$ and set $\gD_{\xi}$ $\df$ $\sqrt{3}$ $-\sqrt{2+2\xi}$.
In case 1) we use the estimate
$1/\gD_{\xi}$ $\geq \sqrt{3}/(1-2\xi)$ for $0<\xi<1/2$ so that from (\ref{newerode})
$$\derz{\xi}
<\frac{\sqrt{6}}{z}\sqrt{1+\xi}\left(1-\frac{1}{C}\frac{\xi}{1-2\xi}\right).$$
Here, $\derz{\xi}<0$ for $1/2>$ $\xi$ $>$ $\xi_1^*$ $\df$ $C/(2C+1).$
Let $c_1$ $=\max\{\xi_0,\xi_1^*\}.$

In case 2) we note that $1/\gD_{\xi}$ $\leq$ $-\sqrt{3}/(2\xi-1)$ for $\xi>1/2.$ We find
$$\derz{\xi}
< \frac{\sqrt{6}}{z}\sqrt{1+\xi}\left(1-\frac{1}{C}\frac{\xi}{2\xi-1}\right)
$$
and $\derz{\xi}<0$ for $1/2<$ $\xi$ $<\xi_2^*$ $\df$ $C/(2C-{1}).$ Let $c_2$ $=$ $\max\{\xi_0,\xi_2^*\}$.
\end{proof}
\begin{thm}\label{thm2}
Suppose $R_z< 0$ on $I=$ $[z_0,\infty)$ with $z_0>0$ and $\xi_0$
$>1/2.$ Then, for $\rho$ $\df\sqrt{3/2},$ there are positive
constants $\ga$, $c_3$ and $c_4$ so that the following holds on $I:$
$$c_3\left((R[z])^{-(\rho-1)}+R[z]\ln\left(\frac{1+z}{1+z_0}\right)\right)\leq M[z]\leq c_4 \left(\frac{1+\ln(1+z)}{(R[z])^{(\rho-1/2)}}\right)^2$$
\end{thm}
\begin{proof}
We first note that since $R_z/\gD_{\xi}$ $>0$ on $I$, we find from
(\ref{newerode}) that $\derz{\xi}$ $>$ $0.$ Let us set
$K_{\xi}\df-\sqrt{3}\sqrt{1+\xi}/\gD_{\xi},$ noting that $\rho\df$
$\sqrt{3/2}<$ $K_{\xi}\leq$ $K_{\xi_0}$ for $\xi>1/2$ is decreasing
as function of $\xi$ and, in turn, also as a function of $z.$
Recalling that $R_z$ $<0,$ we find
$$\derz{\xi}+\rho \xi\frac{R_z}{R[z]}\geq \derz{\xi}+K_{\xi}\xi \frac{R_z}{R[z]}\geq\sqrt{6}\frac{\sqrt{1+\xi_0}}{1+z};$$
$$\frac{d\,}{dz}\left(\xi R^{\rho}\right)\geq \sqrt{6}\sqrt{1+\xi_0}\frac{R^{\rho}[z]}{1+z}
.$$
Now, by the monotonicity of $R[z]$,
\begin{align}
\xi[z] &\geq  R^{-\rho}[z]\left(\xi_0 R^{\rho}[z_0]+\sqrt{6}\sqrt{1+\xi_0}\int_{z_0}^z\frac{R^{\rho}[s]ds}{1+s}\right)
\notag\\
&\geq R^{-\rho}[z]\left(\xi_0 R^{\rho}[z_0]+R^{\rho}[z]\sqrt{6}\sqrt{1+\xi_0}\int_{z_0}^z\frac{ds}{1+s}\right)
\notag
\end{align}
After multiplying through by $R,$ it is clear that we may choose
$c_3$ $\leq$ $\min\{\xi_0 (R[z_0])^{\rho},\sqrt{6}\sqrt{1+\xi_0}\}.$

Now, let us set $\rho_0\df K_{\xi_0}$ and $q_0$ $\df$
$(\xi_0+1)/\xi_0.$ Then, for obvious substitution defining $\xi[s],$
$$ \derz{\xi}+\rho_0 \xi\frac{R_z}{R[z]}\leq\sqrt{6}\frac{\sqrt{1+\xi}}{1+z};$$
$$\frac{d\,}{dz}\left(\xi R^{\rho_0}\right)\leq \sqrt{6}\sqrt{1+\xi[z]}\frac{R^{\rho_0}[z]}{1+z}
$$
\begin{align}
\xi[z] &\leq  R^{-\rho_0}[z]\left(\xi_0 R^{\rho_0}[z_0]+\sqrt{6}\int_{z_0}^z\frac{\sqrt{1+\xi[s]}R^{\rho_0}[s]ds}{1+s}\right)
\notag\\
&\leq R^{-\rho_0}[z]\left(\xi_0 R^{\rho_0}[z_0]+R^{\rho_0}[z_0]\sqrt{6}\sqrt{1+\xi[z]}\int_{z_0}^z\frac{ds}{1+s}\right)
\notag
\end{align}
$$
\sqrt{\frac{\xi[z]}{q_0}}<\frac{\xi[z]}{\sqrt{1+\xi[z]}}\leq R^{-\rho_0}[z]\left(\frac{R^{\rho_0}[z_0]\sqrt{\xi_0}}{\sqrt{1+\xi_0}}
+ \sqrt{6}R^{\rho_0}[z_0]\int_{z_0}^z\frac{ds}{1+s}\right)
$$
$$\xi[z]<q_0R^{2\rho_0}[z_0]R^{-2\rho_0}[z]\left(1+\sqrt{6}\int_{0}^z\frac{ds}{1+s}\right)^2
$$
noting that $\xi^2/(1+\xi)$ $\geq \xi/q_0.$ Choosing $c_4\geq
q_06R^{2\rho_0}[z_0],$ the result follows by multiplying through by
$R$.
\end{proof}

We see that Theorems \ref{thm1} and \ref{thm2} can apply for
$R[z]=$ $R_{\gO_{\gL}}[z]$ (modulo a rescaling factor) as above for certain values of $\gO_{\gL}$: We denote by $I^{\pm}_{\gO_{\gL}}$
the subset of $(0,\infty)$ for which $\pm R_z$ $>0$ and we replace
$C$ by $C^{\pm}$ in the case that (\ref{upprbnd}) holds, respectively.
\begin{remark}\label{rmk2}
For $R_{\gO_{\gL}}[z]$ as in (\ref{R}) we find that when $\gO_{\gL}$
$=1$ there is to every interval of the form $(0,z_2),$ an associated
$C^+$ depending on $z_2>0.$ Moreover, for every $0\leq$ $\gO_{\gL}$
$<1$ there is a $z_{\gL}$ $>0$ where for every positive $z^{\pm}$
with $z^{\pm}\gtrless$ $z_{\gL}$ there is a $C^{\pm}$ associated to
$(0,z^+)$ and $(z^-,\infty),$ respectively. [The singularities $z_{\gL}$ will be discussed in further detail in Section 3.]
\end{remark}
\begin{proof} For $\gO_{\gL}$ $=1$ we find $I^+_1$ $=$ $(0,\infty)$ with $zR_z/R[z]$ $=1/(z+1).$
For $0\leq$ $\gO_{\gL}<1,$ it is not difficult to show that $z/R[z]$
is bounded from below on $(0,\infty)$ by a positive constant,
depending $\gO_{\gL}$. Therefore, the sign of $zR_z/R[z]$ depends on
that of $R_z.$ For $\gO_{\gL}<1$ we find that the sign of $R_z$ is
same as that of $(z+1)\cI(z+1)-$ $\int_1^{z+1}\cI(y)dy$ which is a
monotonically decreasing function of $z$ with a unique positive root
$z_{\gL}>0$, depending on $\gO_{\gL}$. So, $I^+_{\gO_{\gL}}$
$=(0,z_{\gL})$ and $I^-_{\gO_{\gL}}$ $=$ $(z_{\gL},\infty).$ Hence,
for $z^{\pm}$ as above, there are positive constants $C^{\pm}$ so
that $zR_z(z)/R[z]$ $>$ $\pm C^{\pm}$ on intervals $(0,z^+)$ and
$(z^-,\infty),$ respectively.
\end{proof}

In a certain case of interest, we find that for certain initial conditions the growth of $M[z]$
roughly follows that of a power function for large $z$. 
\begin{cor} In the case of Theorem \ref{thm2}
we have for $R=R_{\gO_{\gL}}$ with $0\leq$ $\gO_{\gL}<1$ that, given $M_0>$ $2R[z_0]$ $>0$
and $z_0>z_{\gL},$ for any $\ga>0$ there are positive constants
$k_1$ and $k_2$ so that for $\rho=\sqrt{3/2},$
$$k_1z^{\rho-1}\leq  M[z]\leq k_2 z^{2\rho -1 +\ga}$$
on $I$ $=$ $[z_0,\infty)$.
\end{cor}
\begin{proof}
It is not difficult to show that to any such $\gO_{\gL}$ there are
positive constants $C_1$ and $C_2$ so that
$$C_1/{z}<R_{\gO_{\gL}}[z]<C_2/{z}$$ holds on $I.$ The result immediately follows by Theorem \ref{thm2}.
\end{proof}
We may also conclude
\begin{cor}
If either case 1) or 2) of Theorem \ref{thm1} holds on $I=$ $[z_0,z_1),$
then $r(z)$ and $t(z)$ are both solvable on $I.$ Moreover, $r(z)$ is strictly increasing and $t(z)$
is strictly decreasing on $I$.
\end{cor}
\begin{proof}
We find that $\derz{t}$ and $\derz{r}$ are smooth functions of $z$
since $\xi\neq 1/2$ is smooth. By inspection, we find that
$\derz{t}<0$ on $I$ so that, by our assumption on $\gs$, we see for
$\fh$ as in  (\ref{moreests}) that $\fh$ $\geq 1,$ increasing with $z.$ Then, $\cI_2+\xi\cI_1>0$ for $z\in I$ and,
hence, from (\ref{newerode}) we see that $\derz{r}$ $>0$ for $z\in
I.$
\end{proof}
We note finally that these results are consistent with physical
interpretation where $t$ is interpreted as "look-back" time from an
observer at $r=0$ with a (locally) expanding universe (c.f.
\cite{cpt, kt}).
\section{Study of Singularities, part A: Critical points depending on ${\gO_{\gL}}$}
We now consider how singularities may depend on the parameter
$\gO_{\gL}$ for $R[z]$ $=R_{\gO_{\gL}}[z]$ . As in
Proposition \ref{second}, a singularity arises at $z=z_{\gL}$ where
\begin{equation}\left[R_z\right|_{z=z_{\gL}}=\frac{(1+z_{\gL})\cdot
\cI(1+z_{\gL})- \int_1^{1+z_{\gL}}\cI(y)dy}{(1+z_{\gL})^2}=0.
\label{root}\end{equation}
\begin{prop}\label{zests}
The values $z_{{\gL}}$ satisfy $z_{{\gL}}\geq 1.25,$ increasing
as a continuous function of $\gO_{\gL}$ in the domain $0\leq$ $\gO_{\gL}$ $<1.$ Moreover,
there are positive constants $c_1,c_2$ and $c_3$ so that
$$\left[c_1\ln\left(\frac{1}{1-\gO_{\gL}}\right) +c_2\right]^{1/4}\leq
z_{\gL}+1\leq c_3\frac{1}{1-\gO_{\gL}}$$
$\forall$ $\gO_{\gL}$. Hence, $z_{\gL}$ $\rightarrow$ $+\infty$
as $\gO_{\gL}$ $\rightarrow 1.$
\end{prop}
\begin{proof}
It is not difficult to show from (\ref{root})
that $z_{{\gL}}|_{_{\gO_{\gL}=0}}=1.25$
and that  $z_{{\gL}}$ $>0$ $\forall \gO_{\gL}.$
Now, let us set $q \df$ $1+z_{{\gL}}$
and note that (\ref{root}) gives
$q\cI(q)=\int_1^q\cI(y)dy.$ Implicit differentiation now gives
$$q\frac{d\,q}{d\gO_{\gL}}\frac{\partial\cI(q)}{\partial q}
=\int_1^q\frac{\partial\cI(y)}{\partial\gO_{\gL}}dy-q\frac{\partial\cI(q)}{\partial\gO_{\gL}}.
$$
Applying $q\cI^3(q)$ $=\cI^2(q)\int_1^{q}\cI(y)dy$ on the second term, right-hand side, we compute
$$Q(q)\frac{d\,q}{d\gO_{\gL}}=\frac{1}{\cI^3(q)}\int_1^q\cI(y)K(y,q)dy$$
where
$K(y,q)$ $\df$ $\cI^2(q)(q^3-1)$ $-\cI^2(y)(y^3-1)$ and
$Q(q)\df 3q^3(1-\gO_{\gL}).$
Since $\cI^2(y)(y^3-1)$ is strictly increasing as a function of $y\geq 1,$
we find $K(y,q)> 0$ for $1\leq y$ $<q$. Thus, $\frac{d\,\,q}{d\gO_{\gL}}>0$ $\forall \gO_{\gL}$
and hence $q\geq 2.25$ $\forall \gO_{\gL}.$
For $k_1\df\frac{4}{3}\int_1^{2.25}\cI(y)K(y,2.25)dy$ we estimate
$$Q(q)\frac{d\,q}{d\gO_{\gL}}\geq \frac{3k_1/4}{\cI^3(q)}\geq 3k_1/4$$
so that
$$4\int_{2.25}^q{y^3dy}\geq
k_1\int_{0}^{\gO_{\gL}}\frac{dx}{1-x},$$
and our choices of $c_1$ and $c_2$ are clear since
$$q^4\geq k_1 \ln\left(\frac{1}{1-\gO_{\gL}}\right)+2.25^4$$

Next we note that
$$q=\frac{1}{\cI(q)}\int_1^q\cI(y)dy \geq \sqrt{(1-\gO_{\gL})q^3}\int_1^q\frac{dy}{\sqrt{1+y^3}}$$
so that $\sqrt{q}\leq \frac{k_2}{\sqrt{1-\gO_{\gL}}}$
with $1/k_2\df\int_1^{2.25}\frac{dy}{\sqrt{1+y^3}}.$
We choose $c_3=k_2^2$ and we are done.
\end{proof}

With a broad range of values $z_{\gO_{\gL}},$ bound by
the estimates of Proposition \ref{zests},
one may expect difficulties
in applying the present work to cosmological models - with
singularities $z_{\gO_{\gL}}$ well within
observed redshift values \cite{c,cpt,cr}.
However, some such singularities may conceivably be of type $0/0$ if
both $R_z$ and $\sqrt{1+2E[z]}$ $-$ $\sqrt{2M[z]/R[z]+2E[z]}$
were to have zeros of identical order,
rendering the singularities, in some
sense, removable. We demonstrate such a case in the next section. 
\section{Study of Singularities, part B: FRW Model}\label{studyB}
Using solutions from the well-known Freedman-Robertson-Walker model, we analyze our map $(E,D_L,R_0)$
$\rightarrow$ $(r,t,M)$ and study
singularities of the system (\ref{ode}) and their dependence on $\gO_{\gL}$.
We restrict the map as follows:
We fix
the function $E(r)$ and restrict $D_L(z)$ and $R_0(r)$ to certain
one-parameter classes in the pre-image space; and, we fix the function $M(r)$ in the image space. Here, we consider data given by
$R[z]$ $=R_{\gO_{\gL}}[z]$ as in (\ref{R}) and we set
\begin{equation} E=\frac{r^2}{2}, M=\frac{r^3}{2},R_0 = cM/E=cr\label{case}\end{equation} for
parameter $c> 0.$ Following \cite{aag}, we have $R(r,t)=r\cdot a(t)$
where for some (real) parameter $\eta$ with $k_c$ $\df$
$\sqrt{c+c^2},$
\begin{align}\label{frwsol}
a(t)&=\frac{\cosh{\eta}-1}{2}+(c\cosh{\eta}+k_c\,\sinh{\eta})&\df\fF_c(\eta)\\
t&=\frac{\sinh{\eta}-\eta}{\sqrt{2}}+\sqrt{2}(c\sinh{\eta}+k_c\,\cosh{\eta})&\df\fG_c(\eta)\notag
\end{align}
Here, $\eta$ is known as "conformal time" which in our case depends
on $a$ and $t$ by
$\eta=\eta(t)=\int_{\sqrt{2}k_c}^t\frac{d\gt}{\sqrt{2}a(\gt)}.$ We
note that $\fF_c$ and $\fG_c$ are each invertible for $\eta$ on an
open interval containing $0.$ In particular, $\fF_c$ is invertible
for
$\eta$ $>-\mbox{arctanh}(2k_c/(1+2c))$ and $\fG_c$ is invertible
where $a>0,$ so that $a(t)$
$=\fF_c\circ\fG^{-1}_c(t)$ indeed holds
for $t$ in a neighborhood containing $k_c.$
Moreover, using (\ref{dt}) and setting $c=$ $a(t_0)$ with
$t_0$ $\df$ $t(z_0)$ $=\sqrt{2}k_c$ for some $z_0>0,$
$$\derz{t}=\frac{-a(t)}{(1+z)\dot{a}(t)};\,\,a[z]=a(t(z))=c\frac{1+z_0}{1+z}.$$

Given $R[z],$ we find, indirectly, the resulting solutions of
(\ref{ode}):
\begin{align}
t(z)&=\fG_c(\fF_c^{-1}(a[z]))\label{sols}\\
r(z)&=R[z]/a[z]=\frac{\int_1^{1+z}\cI(y)dy}{(1+z_0)c}\notag\\
M[z]&=\frac{1}{2}\left(\frac{\int_1^{1+z}\cI(y)dy}{(1+z_0)c}\right)^3
\notag
\end{align}
As for the relevance of this case to physical models, we note that the associated energy density $\rho[z]$ is a smooth function on $(0,\infty).$

We are ready to state
\begin{thm}\label{sol}
For any given $0\leq$ $\gO_{\gL}\leq 1$ and $z_0>0$ there is a smooth function $R_0(r)$ so that $E(r),$ $M(r)$ as in (\ref{case}) and $R[z]$ $=R_{\gO_{\gL}}[z],$ the system (\ref{ode}) with initial
conditions
\begin{equation}\vec{X}(z_0)=\left(\begin{matrix}R[z_0]/c
\\\sqrt{2}k_c\\ (R[z_0]/c)^{3}/2\end{matrix}\right)\label{initcond}\end{equation}
has a smooth solution $\vec{X}$ on an open interval $I\ni$ $z_0.$
\end{thm}
\begin{proof}
For those $\eta$ where the solutions (\ref{frwsol}) hold we
also have $R^{\prime}$ $=a(t)$ $>0$ and $\dot{R}^{\prime}$
$=\dert{a}$ $=\dereta{a}/\dereta{t}$ $>0.$ Since the initial
conditions hold for $\eta=$ $\fF_c^{-1}(a(z_0))$ $=0,$ (\ref{sols})
also holds for $\eta$ in some interval containing $0$.
From continuity arguments we see there is also some open interval $I$ $\ni z_0$ on which such
solutions $\vec{X}(z)$ in turn hold.
\end{proof}
We note that the above method provides no solutions for $r(z)$ and $M[z]$ in
the case $R_0\equiv 0$ unless more data is prescribed, such as
asymptotic conditions for the ratio $R/c$ in terms $z$ and
$z_0$ (c.f. Example A, p. 5 \cite{cr}).  Moreover, we note that singularities
may occur in the form $R=2M$ and/or $\dot{a}=0$ away from $z_0$ so that we
may not arbitrarily extend the domain $I$ of the solution via Proposition \ref{third}.

We may apply Proposition \ref{third} in regards to uniqueness of solution: To rule out one type of singularity,
we compute $E^{\prime} \cG R$  $-\cF$ via (\ref{gerf}). First, we set
$\xi=r/(2a(t))$ and $\xi_0=r/(2c)$ and compute
\begin{align}
\frac{E^{\prime}\cdot(J_E-RJ_M)}{J_R}=&
\sqrt{E+\xi}\frac{RE^{\prime}}{2}\int_1^{c/a(t)}\frac{s^{1/2}(s-1)ds}{(Es+\xi)^{3/2}}\notag\\
=&-\frac{1}{2}\int^{c}_{a(t)}\frac{\sqrt{\gt}
(\gt-a(t))}{(\gt+1)^{3/2}}d\gt\sqrt{\frac{a(t)+1}{a(t)}}\leq 0\notag
\end{align}
for $c,a(t)>0$.
We now compute,
\begin{align}
\frac{R_0^{\prime}J_{R_0}}{J_R}&= -c\cdot\sqrt{\frac{R_0}{ER_0+M}}/\sqrt{\frac{R}{ER+M}}\notag\\
&= -c\cdot\sqrt{\frac{c}{c+1}}\sqrt{\frac{a(t)+1}{a(t)}}\notag
\end{align}
which is strictly negative. Therefore, $E^{\prime} \cG R$  $-\cF$ $< 0$ and we have ruled out
case 2) of Proposition \ref{first}. Knowing also that $\dot{R}^{\prime}[z]|_{z=z_0}$ $\neq 0$ in this case we state
\begin{thm}\label{solsb}
The solutions of Theorem \ref{sol} are unique for $z_0$ $\neq$ $z_{\gL}.$
\end{thm}
The solutions (\ref{ode}) stand in glaring contrast to the result of Proposition \ref{first}:
Indeed, we note that the right-hand sides of equations (\ref{dr})
and (\ref{dt})
under the conditions of Theorem \ref{sol} have no positive
singularities $z_{\gL}$ as $E^{\prime}$ $+M^{\prime}/R$
$-MR^{\prime}/R^2$ $=r+r/a>0;$ yet, we find that the determinant of $\cU$
in (\ref{sys}) vanishes at $z=z_{\gL}$.
Since Theorem \ref{sol} applies in the case $z_0$ $=z_{\gL},$
one may suspect that these
singularities are, in some sense, removable - so we shall see
in remainder of this section.

We give specific cases, depending on $R_0$, in which the solutions
$\vec{X}$ can be smoothly extended across singularities
$z=z_{{\gL}}.$ For such solutions to be valid, it suffices that
$\dot{a}[z]>0$ is smooth, that (\ref{root}) holds, and that as in
(\ref{chain}) $R[z_{\gL}]$ $=$ $2M[z_{\gL}]$ (or perhaps as smooth
extensions defined at $z_{\gL}$). Then,
$$R^2[z_{\gL}]=\frac{((1+z_0)c)^3}{(1+z_{\gL})^3}
$$
and, hence,
$$\cI(1+z_{\gL})=R[z_{\gL}]=\frac{((1+z_0)c)^{3/2}}{(1+z_{\gL})^{3/2}}.$$
From this we obtain the corresponding value of $c$ by which we define
\begin{equation}c_{{\gL}}\df\frac{1+z_{\gL}}{
(1+z_{0})(\gO_{\gL}+(1-\gO_{\gL})(1+z_{\gL})^3)^{1/3}}.
\label{cL}\end{equation}

We are ready to state
\begin{thm}\label{smooth} Under the hypothesis of Theorem \ref{sol}
for every $z_0>0$ and $0\leq$ $\gO_{\gL}$ $<1$ there is a smooth $R_0(r)$ for which
the resulting solution $\vec{X}(z)$ with initial conditions (\ref{initcond}) can be uniquely extended to be of class $C^{\go}((0,\infty)).$
\end{thm}
\begin{proof}
We take $R_0(r)$ $=c_{_{{\gL}}}r$ for $c_{_{{\gL}}}$ as in
(\ref{cL}). Using (\ref{frwsol}) and following the Chain Rule
formula
\begin{equation}\dot{a}(t(z))=\frac{\derz{a[z]}}{\derz{\fG_{c{_{\gL}}}(\eta(z))}}
=\frac{\derz{a[z]}}{\sqrt{2}\fF_{c_{_{\gL}}}(\eta(z))\derz{\eta[z]}}
=\frac{\derz{a[z]}}{\sqrt{2}a[z]\derz{\eta[z]}},
\label{adot}\end{equation} it suffices to
show that $\eta [z]$ is smooth and that $\derz{\eta}$ is strictly
positive on $(0,\infty).$ To do this, we set
$$\eta[z]= - \int \frac{dt[s]}{\sqrt{2}a[s]},$$
with $dt[z]\df\derz{t(z)}dz,$ and we proceed to analyze the
integral. We may write
\begin{equation}
\label{dtbydz}
\derz{t}=H(z)\frac{-R_z}{2M-R}=H(z)\frac{-R_z}{r(z)(r(z)^2-a(z))}\end{equation}
for some real-valued function $H$ $>0,$ analytic for $z>0.$ Here
$r^2-a$ is an increasing function which vanishes at $z_{\gL}$ and is
of the same sign as that of $-R_z$ $\forall$ $z>0.$

Now, we check the behavior of $\derz{t}$ near the singularity, applying analyticity arguments as follows: Using (\ref{sols}) and (\ref{root}) we compute
\begin{align}
&\left[\frac{d^2R[z]}{dz^2}\right|_{z=z_{\gL}}=-\frac{3(1-\gO_{\gL})}{2}(1+z_{\gL})\cI^3(1+z_{\gL})<0\notag\\
&\left[\derz{M[z]}\right|_{z=z_{\gL}}=\frac{3\cI(1+z_{\gL})}{2(1+z_{\gL})}\df \fM_{\gL}> 0;\notag
\end{align}
and, in turn, we find that
$$2M[z]-R[z] =2\fM_{\gL}\cdot(z-z_{\gL})+\fP(z)(z-z_{\gL})^2
$$ for some analytic function $\fP.$
Here, $\frac{d^2 R[z]}{dz^2}$ $<0$ on a neighborhood of $z_{\gL}$ where $R_z$ has a zero of order exactly 1.
We therefore find that the following limit exists as we compute:
\begin{align}\lim_{z\rightarrow z_{\gL}}\frac{-\derz{R}}{2M[z]-R[z]}
&=\left[-\frac{\frac{d^2 R[z]}{d z^2}}{2\derz{M[z]}}\right|_{z=z_{\gL}}\notag\\
&=\frac{(1-\gO_{\gL})(1+z_{\gL})\cI(1+z_{\gL})}{2}>0\notag
\end{align}
We may conclude therefore that $\frac{-R_z}{2M[z]-R[z]}$ extends to
an analytic, positive-valued function on $(0,\infty)$ and, hence,
$C^{\go}((0,\infty))$ $\ni$ $\derz{\eta[z]}>0.$ Therefore,
$\dot{a}[z]$ is well-defined and is non-zero; and, moreover,
$\eta[z]$ is of class $C^{\go}((0,\infty))$. The uniqueness follows
since Theorem \ref{sol} applies to any open interval not containing
$z_{\gL}$\end{proof}
\begin{remark}\label{sing} Following the calculations in the proof of
Theorem \ref{smooth}, we note that any other choice of positive
$c\neq c_{\gL}$ leads to a singularity of order one at $z=z_{\gL}$
for $\dot{R}^{\prime}[z]$ $=\dot{a}[z]$ as evident in (\ref{adot})
and (\ref{dtbydz}). This gives singularities in equations
(\ref{dr}), (\ref{dt}), and (\ref{Rprime}), and renders the
resulting system (\ref{ode}) invalid at $z_{\gL}$.
\end{remark}
\begin{remark} \label{opt} Our FRW model is consistent with the construction of
$R[z]$ as in \cite{cpt} where $\frac{R_0}{R}$ $=1+z$ with no
prescribed value of $c$. Moreover, our choice of $c=c_{\gL}$ is
optimal in assuring the largest possible domain of
$C^{\go}$-solvability.
\end{remark}
\section*{Discussion}
We make several concluding comments and a conjecture: First, we note
that in the case of Theorem \ref{smooth} the various right-hand
sides of the system (\ref{ode}) can each be written in the form
$\cA(z) +\cB(z)\frac{R_z}{R[z]-2M[z]}$ for smooth functions $\cA$
and $\cB.$ Thus, the arguments for the smooth extension of
$\derz{t}$ beyond the critical points $z_{\gL}$ of $R[z]$ also apply
to $\derz{r}$ (also, of course with circular reasoning, to
$\derz{M}$). However, in the general mapping scheme (\ref{map}), we
have no way to predict the order of the zeros of $\derz{M}$ nor any
a priori justification to expect these singularities to be removable
- not even as we fix our choice of $R[z]$ $=R_{\gO_{\gL}}[z]$.

Second, one may interpret the removability or non-existence of such
singularities as indication of compatibility of the corresponding
models as one imposes $R[z]$ on a model that prescribes $E$ and
$R_0$. (Here the LTB model would be said to 'mimic' the given
cosmological constant model, c.f. \cite{ah}.) Applying such criteria
to Remark \ref{sing} one does not expect every LTB model to be
compatible with such a cosmological-constant model (at least not for
$z$ near $z_{\gL}$). However, from Theorem \ref{smooth} we do find,
as a check of our analysis, that the cosmological-constant models
for $0\leq\gO_{\gL}< 1$ are each compatible with at least one
LTB/FRW model: Our choice of $R_0$ identifies an optimal FRW model,
in the sense of Remark \ref{opt}.

Finally, one conjectures that these removable, $0/0$-type
singularities may yet lead to instability of numerical solutions of
the system (\ref{ode}) (but here at certain finite $z$ (!) c.f.
\textsection IV \cite{cr}). Such investigations are beyond the scope
of the present work.

\end{document}